\documentclass[12pt]{amsart}

\newtheorem{theorem}{Theorem}[section]
\newtheorem{lemma}[theorem]{Lemma}

\usepackage[dvips]{graphics}

\theoremstyle{definition}

\theoremstyle{remark}
\newtheorem{remark}[theorem]{Remark}

\theoremstyle{conjecture}

\theoremstyle{corollary}
\newtheorem{corollary}[theorem]{Corollary}

\theoremstyle{problem}

\numberwithin{equation}{section}

\large\normalsize

\begin{document}
\title{ Perfect matchings of line graphs with small maximum degree}

\author{Weigen Yan}
\address{School of Sciences, Jimei University,
Xiamen 361021, China} \email{weigenyan@263.net}
\thanks{The first author was supported in part by NSFC
Grant (10771086) and by Program for New Century Excellent Talents in
Fujian Province University.}
\author{Fuji Zhang}
\address{School of Mathematical Science, Xiamen Univiversity,
Xiamen 361005, China} \email{fjzhang@xmu.edu.cn}
\thanks{The second author was supported in part by NSFC Grant \#10831001.}



\begin{abstract}
Let $G$ be a connected graph with vertex set $V(G)=\{v_1,v_2,\cdots,v_{\nu}\}$
, which may have multiple edges but have no loops, and $2\leq d_G(v_i)\leq 3$ for $i=1,2,\cdots,\nu$,
where $d_G(v)$ denotes the degree of vertex $v$ of $G$. We show
that if $G$ has an even number of edges, then the number of perfect matchings of the line
graph of $G$ equals $2^{n/2+1}$, where $n$ is the number of 3-degree vertices of $G$.
As a corollary, we prove that
the number of perfect matchings of a connected cubic line graph with $n$ vertices
equals $2^{n/6+1}$ if $n>4$, which implies the conjecture by Lov\'asz and Plummer holds for the connected cubic
line graphs.
As applications, we enumerate perfect matchings of the
Kagom\'e lattices, $3.12.12$ lattices, and Sierpinski gasket with dimension two in the context of statistical
physics.
\end{abstract}
\maketitle
\section{Introduction}
\hspace*{\parindent}
Throughout this paper, we suppose that $G=(V(G),E(G))$ is a
connected graph with the vertex set $V(G)=\{v_1,v_2,\cdots,v_{\nu}\}$
and the edge set $E(G)$ which may have multiple edges but have no
loops, if not specified. The line graph of $G$, denoted by $L(G)$,
is defined as the graph whose vertex set $V(L(G))=E(G)$ and two
vertices $e$ and $f$ in $L(G)$ are joined by $i$ ($i=0,1,2$) edges
if and only if two edges $e$ and $f$ in $G$ have $i$ end vertices in
common. A perfect matching of $G$ is a set of independent edges of
$G$ covering all vertices of $G$. Denote the number of perfect
matchings of $G$ by $M(G)$. It is well known that computing $M(G)$
of a graph $G$ is an $NP$-hard problem (see \cite{Jerr87,Lova86,Valiant79}).

Let $G$ be a graph with $\nu$ vertices and let $G_0,G_1,\cdots,G_k$ be
graphs such that $G_0=G$ and, for each $i> 0$, $G_i$ can be obtained
from $G_{i-1}$ by subdividing an edge once. Then $G_k$ is said
to be a subdivision of $G$. For convenience, we can regard $G$
as a subdivision of $G$. Let $S(G)$ denote the graph obtained from
$G$ by subdividing every edge once.

A classical theorem of Petersen \cite{Petersen1891} asserts that
every cubic graph without a cut-edge has at least a perfect
matching. This result can be derived as a corollary of Tutte's
1-factor theorem \cite{Lova86}. Edmonds, Lov\'asz, and Pulleyblank
\cite{ELP82} and Naddef \cite{Naddef82} proved that each cubic graph
with $\nu$ vertices has at least $\frac{\nu}{4}+2$ perfect matchings if it has no cut-edge
and has at least $\frac{\nu}{2}+1$ perfect matchings if it is
cyclically 4-edge-connected. Similar bounds can be achieved using
Lov\'asz' matching lattice theorem \cite{Lova87}. Recently, Kral,
Sereni, and Stiebitz \cite{KSS} improved the lower bound for graphs
without a cut-edge to $\frac{\nu}{2}$. In the 70's, Lov\'asz and Plummer
conjectured (in the mid-1970¡¯s) that every cubic graph with no
cut-edge has exponentially many perfect matchings.
Voorhoeve \cite{Voor79} proved this conjecture for bipartite cubic
graphs. He proved that every cubic bipartite graph $G$ with no cut-edge has at least
$6\cdot(\frac{4}{3})^{\frac{\nu}{2}-3}$ perfect matchings, where $\nu$
is the number of vertices of $G$. Furthermore, Lov\'asz and Plummer \cite{Lova86}
conjectured that for $k\geq 3$ there exist constants $c_1(k)>1$ and $c_2(k)>0$ such
that every $k$-regular elementary graph (i.e., 1-extendable graph)
with $2\nu$ vertices contains at least $c_2(k)c_1(k)^{\nu}$ perfect matchings.
Schrijver \cite{Schr98} proved this conjecture for the $k$-regular bipartite graphs.
He poved that for $k>2$ every $k$-regular bipartite graph $G$ with $2\nu$ vertices has at
least $\nu!\left(\frac{k}{\nu}\right)^n$$ (\geq
\sqrt{2\pi}\left(\frac{\nu}{e}\right)^{\nu}$ perfect matchings.

The above conjectures have been proved to
be challenging questions, and are still open. As far as we know the
latest result of the first conjecture for the case of the non bipartite
graphs was achieved by Chudnovsky and Seymour \cite{CS09} who proved that
every cubic planar graph $G$ with no cut-edge has at least
$2^{\nu/655978752}$ perfect matchings, where $\nu$ is the number
of vertices (with degree three) of $G$.

Inspiring by these discussions, we attempt to find some kind of graphs the number of perfect
matchings of which can be determined by the number of vertices of degree three.
In the next section, we prove that
if $G=(V(G),E(G))$ is a connected graph with an even number of edges
, which may have multiple edges but have no loops, and satisfies $2\leq d_G(v)\leq 3$, for $v\in V(G)$, then the
number of perfect matchings of the line graph of $G$ equals
$2^{n/2+1}$, where $n$ is the number of 3-degree vertices of $G$. As a corollary, we also prove that
the number of perfect matchings of a connected cubic line graph with $n$ vertices
equals $2^{n/6+1}$ if $n>4$.
As applications, in Section 3 we enumerate perfect matchings of the
Kagom\'e lattices, $3.12.12$ lattices, and Sierpinski gasket with dimension two in the context of statistical
physics. Finally, in Section 4 we give some remarks.
\begin{figure}[htbp]
  \centering
 \scalebox{1}{\includegraphics{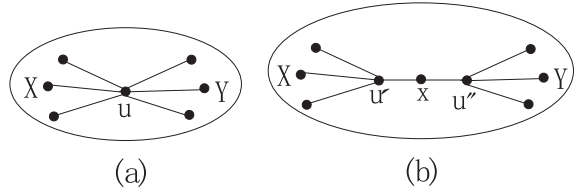}}
  \caption{\ (a)\ A graph $G$;
  \ (b)\ the corresponding graph $G'$.}
\end{figure}
\section{Main results} \hspace*{\parindent}
We first introduce some lemmas. Let $G$ be a graph and $u$ a vertex of $G$. Let $X\cup Y$ be a
partition of the edges incident $u$. For an edge $e$ incident to $u$,
let $\phi(e)$ be the other
endpoint of $e$. Construct a new graph $G'$ from
$G$ as follows (see Figure 1):

{\it (i)}\ remove $u$ and the incident edges and insert three new vertices $u', u''$ and $x$;

{\it (ii)}\ connect $x$ to $u'$ and $u''$ by an edge, and for $e\in X$, connect $u'$ to $\phi(e)$ by an edge,
and for $e\in Y$, connect $u''$ to $\phi(e)$ by an edge.

For convenience, we say that $G'$ is obtained from $G$ by splitting vertex $u$.
The following lemma is immediate from Lemma 1.3 in \cite{Ciucu97}.

\begin{lemma}
Let $G$ and $G'$ be the graphs defined as above. Then
$$M(G)=M(G').$$
\end{lemma}
\begin{figure}[htbp]
  \centering
 \scalebox{1.2}{\includegraphics{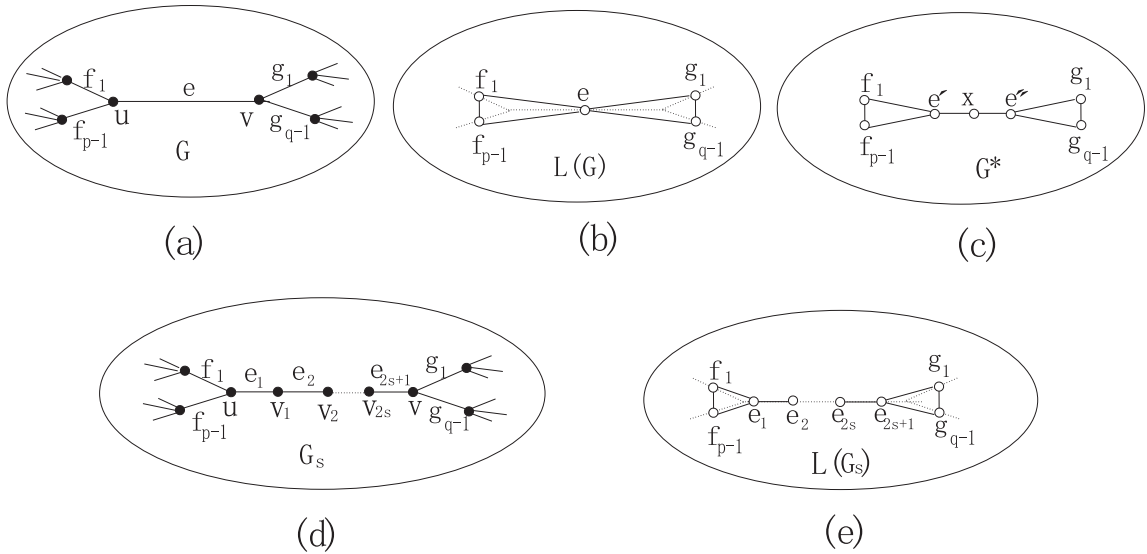}}
  \caption{\ (a)\ A graph $G$;
  \ (b)\ the line graph $L(G)$ and $G$ (dotted lines);
  \ (c)\ the graph $G^*$;
  \ (d)\ the graph $G_s$;
  \ (e)\ the line graph $L(G_s)$ and $G_s$ (dotted lines).}
\end{figure}

\begin{lemma}
Suppose $G$ is a connected graph and $|E(G)|$ is even.
Let $e=(u,v)$ be an edge of $G$ and $d_G(u)\geq 2, d_G(v)\geq 2$.
For any non negative integer $s\geq 0$, let $G_s$ be the graph obtained
from $G$ by subdividing edge $e$ $2s$ times ($G_0=G$). Then

$(a)$\ $M(G)=M(G_s)$;

$(b)$\ $M(L(G))=M(L(G_s))$ for $s=0,1,2,...$.
\end{lemma}
\begin{proof}
From Lemma 2.1, $(a)$ is immediate. Hence it suffices to prove $(b)$.

Suppose $d_G(u)=p\geq 2, d_G(v)=q\geq 2$. Let $e,f_1, f_2, \cdots$, and $f_{p-1}$ be the $p$ edges incident with
vertex $u$, and let $e,g_1, g_2,\cdots$, and $g_{q-1}$ be the $q$ edges incident with vertex $v$
(see Figure 2(a), where $p=q=3$).
Hence $(e,f_1), (e,f_2), \cdots,(e,f_{p-1}), (e,g_1), \cdots, (e,g_{q-1})$
are $p+q-2$ edges in $L(G)$ (see Figure 2(b)). Let $G^*$ be the graph obtained from $L(G)$ by splitting
vertex $e$, which is illustrated in Figure 2(c). By Lemma 2.1,
$$M(L(G))=M(G^*).\eqno{(1)}$$

Since $G_s$ is the graph obtained from $G$ by subdividing edge $e$ $2s$ times,
denote these $2s$ subdividing vertices by $v_1,v_2,\cdots,v_{2s}$ in turn (see Figure 2(d)).
Let $v_0=u$ and $v_{2s+1}=v$, and $e_i=(v_{i-1},v_i), i=1,2,\cdots,2s+1$.
Obviously, $e_1-e_2-\cdots-e_{2s+1}$ is a path with $2s+1$ vertices in $L(G)$
and $d_{L(G)}(e_i)=2$ for $i=2,3,\cdots,2s$ (see Figure 2(e)). Clearly,
$$M(L(G_s))=M(G^*). \eqno{(2)}.$$

Then, by $(1)$ and $(2)$, $(b)$ holds.
\end{proof}

Similarly, we can prove the following:
\begin{lemma}
Suppose $G$ is a connected graph. Let $e=(u,v)$ be an
edge of $G$ and $d_G(u)\geq 2, d_G(v)\geq 2$.
For any non negative integer $s\geq 0$, let $G^{(s)}$ be the graph obtained
from $G$ by subdividing edge $e$ $2s+1$ times. Then
$$M(L(G^{(1)}))=M(L(G^{(2s+1)})), s\geq 1.$$
\end{lemma}

Now we state and prove our main result as follows.
\begin{theorem}
Let $G$ be a connected graph with vertex set $V(G)$
, which may have multiple edges but have no loops, and $2\leq d_G(v)\leq 3$, for $v\in V(G)$.
If $G$ has an even number of edges, then the
number of perfect matchings of the line graph of $G$ equals
$2^{n/2+1}$, where $n$ is the number of 3-degree vertices of $G$.
\end{theorem}

\begin{proof}
 We prove the theorem by induction on the number of 3-degree vertices of $G$.
If $G$ has no 3-degree vertex, i.e., $n=0$, then $G$ is a cycle with an even number of edges.
Hence $M(L(G))=2^{0/2+1}=2$.

Note that $G$ has an even number of 3-degree vertices.
Now we assume that $n\geq 2$. Suppose $u$ and $u'$ are two 3-degree vertices of $G$. Since $G$ is connected,
there exists one path $P(u-u')$: $u=v_0-v_1-\cdots-v_k=u'$ in $G$. Let $j+1=\min\{i|d_G(v_i)=3, 1\leq i\leq k\}$
and $v=v_{j+1}$. Then $1\leq j+1\leq k$. Hence there exists one path $P(u-v)$: $u=v_0-v_1-\cdots-v_{j+1}=v$ in $G$ such that
$d_G(u)=d_G(v)=3$, and $d_G(v_1)=d_G(v_2)=\cdots=d_G(v_{j})=2$ if $j>0$.
By Lemmas 2.2 and 2.3, it suffices to consider two cases $j=0$ and $j=1$.


{\bf Case 1}\ \ $j=0$.

Obviously, $e_1=(u,v)$ is an edge of $G$. If $G$ has three multiple edges connecting vertices
$u$ to $v$, then $G$ has exactly three edges.
This is a contradiction with $|E(G)|$ is even. Hence we need to consider the following two subcases:
\begin{figure}[htbp]
  \centering
 \scalebox{1.1}{\includegraphics{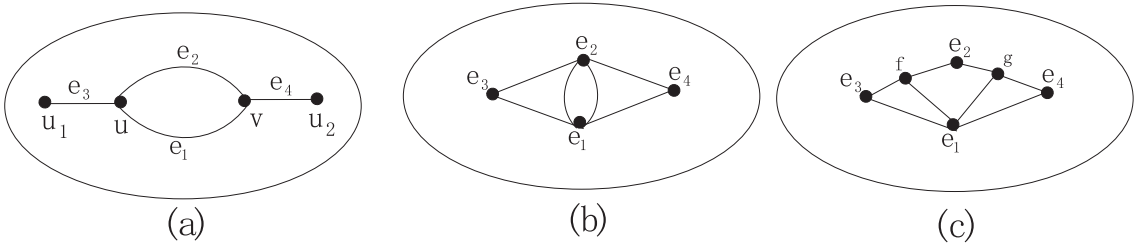}}
  \caption{\ (a)\ A graph $G$;
  \ (b)\ the line graph $L(G)$;
  \ (c)\ the graph $G''$.}
\end{figure}\

{\bf Subcase 1.1}\ \ $G$ has two multiple edges $e_1$ and $e_2$ connecting vertices $u$ and $v$ (see Figure 3(a)).

Let $e_1, e_2$, and $e_3=(u,u_1)$ (resp., $e_1, e_2$,
and $e_4=(u,u_2)$) be the three edges incident with vertex $u$ (resp., vertex $v$) in $G$ (see Figure 3(a)),
and $u_1\neq v, u_2\neq u$.
Construct a new graph $G'$ from $G$
by deleting vertex $v$ and connecting $u$ and $u_2$ by an edge. Hence $G'$ is a connected graph
the degree of each of whose vertices is two or three. Particularly, the number of the 3-degree
vertices of $G'$ equals $n-2$, where $n$ is the number of the 3-degree vertices of $G$. By induction,
$$M(L(G'))=2^{(n-2)/2+1}=2^{n/2}. \eqno{(3)}$$

Note that $d_{L(G)}(e_1)=d_{L(G)}(e_2)=4$
(see Figure 3(b)). Let $G''$ be the graph obtained from $L(G)$ by splitting vertex $e_2$,
which is illustrated in Figure 3(c). By Lemma 2.1,
$$M(L(G))=M(G'').\eqno{(4)}$$
Note that $M(G'')=M(G''-f-e_2)+M(G''-g-e_2)$ (see Figure 3(c)). Since each perfect matching of $G''-f-e_2$ (resp.
$G''-g-e_2$) contains no edge $(e_1,e_4)$ (resp. $(e_1,e_3)$),
by the definition of $G'$, it is not difficult to show that
$$M(G''-f-e_2)=M(G''-g-e_2)=M(L(G')).$$ Hence
$$M(G'')=2M(L(G')). \eqno{(5)}$$
By $(3),(4)$, and $(5)$, $M(L(G))=2^{n/2+1}$.

{\bf Subcase 1.2}\ \ There exists only one edge $e_1$ connecting $u$ to $v$ in $G$.

Let $e_1=(u,v), e_2=(u,u_1)$, and $e_3=(u,u_2)$ (resp., $e_1=(u,v), e_4=(v,u_3)$,
and $e_5=(v,u_4)$) be the three edges incident with vertex $u$ (resp., vertex $v$) in $G$,
and $v\notin \{u_1,u_2\}, u\notin\{u_3,u_4\}$.

{\bf Subcase 1.2.1} $e_1$ is a cut edge of $G$.

If $e_1$ is a cut edge of $G$ then $G-e_1$ has two connected components $G_1$ and $G_2$.
Assume that $G_1$ contains vertex $u$ and $G_2$ contains vertex $v$. Note that $|E(G)|=|E(G_1)|+|E(G_2)|+1$ is even.
Without loss of generality, we suppose that $|(E(G_1)|$ is even and $|E(G_2)|$ is odd. Let $G_3=G[V(G_2)\cup\{u\}]$,
i.e., $G_3$ is the graph obtained from $G$ by deleting all vertices in $V(G_1)\backslash\{u\}$. Note that $e_1$
is a cut vertex of $L(G)$. Hence
$$M(L(G))=M(L(G_1))M(L(G_3)).\eqno{(6)}$$

Let $n_1$ and $n_2$ be the numbers of 3-degree vertices in $G_1$ and $G_3$, respectively. So
$$n_1+n_2=n-1.\eqno{(7)}$$
By induction,
$$M(L(G_1))=2^{n_1/2+1}.\eqno{(8)}$$

In order to enumerate perfect matchings of $G_3$, we first prove the following:

{\bf Claim 1}\ \ Let $H$ be a connected graph with an even number of edges and $u$ a vertex of $G$ satisfying $d_H(u)=1$.
Suppose $e=(u,v)$ is the edge incident with $u$ in $H$ and there exists only three edges
$e=(v,u),e_1=(v,v_1)$, and $e_2=(v,v_2)$ incident with $v$ in $H$ (i.e., $d_H(v)=3$) satisfying $|\{v_1,v_2\}|=2$
(i.e., $e_1$ and $e_2$ are not two multiple edges of $G$).
Construct a new graph $H'$ from $H$ by deleting vertices $u$ and $v$
and connecting vertices $v_1$ to $v_2$ by an new edge
$e'=(v_1,v_2)$. Then $M(L(H))=M(L(H'))$.

In fact, by the definition of $H$, $d_{L(H)}(e)=2$ and $(e_1,e_2)$ is an edge of $L(H)$ which can not be an edge
of a perfect matching of $L(G)$. Hence $M(L(H))=M(L(H)-(e_1,e_2))$. Let $H^*$ be the graph obtained from $L(H)-(e_1,e_2)$
by deleting vertex $e$ and identifying vertices $e_1$ and $e_2$ (the new vertex is denoted by $e^*$).
Obviously, $L(H)-(e_1,e_2)$ is the graph obtained from $H^*$ by splitting vertex $e^*$. By Lemma 2.1,
$$M(H^*)=M(L(H)-(e_1,e_2)).$$
By the definition of $H'$, $L(H')=H^*$. Hence
$$M(L(H))=M(L(H'))$$
implying the claim holds.

Note that $G_3$ satisfies $d_{G_3}(u)=1$ and $d_{G_3}(v)=3$. The three edges incident with $v$ in $G_3$ are
$(v,u), (v,u_3)$, and $(v,u_4)$. We consider the following cases (a) and $(b)$.

{\bf Subcase 1.2.1(a)} $(v,u_3)$ and $(v,u_4)$ are not two multiple edges in $G_3$ (i.e., $u_3\neq u_4$).

Construct a new graph $G_3'$ from $G_3$ by deleting vertices $u$ and $v$ and connecting vertices
$u_3$ to $u_4$ by a new edge $(u_3,u_4)$.
By the claim above,
$$M(L(G_3))=M(L(G_3')).$$
Note that the number of 3-degree vertices of $G_3'$ equals $n_2-1$. By induction,
$$M(L(G_3'))=2^{(n_2-1)/2+1}=M(L(G_3)).\eqno{(9)}$$ Hence, by $(6)-(9)$, we have
$$M(L(G))=2^{n/2+1}.$$

{\bf Subcase 1.2.1(b)} $(v,u_3)$ and $(v,u_4)$ are two multiple edges in $G_3$ (i.e., $u_3=u_4$).

Construct a new graph $G_3^*$ from
$G_3$ by replacing edge $(v,u_3)$ in $G_3$ (resp., $(v,u_4)$)
with a path $(v-w_1-w_2-u_3)$ (resp., $(v-w_3-w_4-u_3)$).
That is, $G_3^*$ is the graph obtained from $G_3$ by
subdividing each of edges $(v,u_3)$ and $(v,u_4)$ twice. By Lemma 2.2,
$$M(L(G_3))=M(L(G_3^*)).$$
From Subcase 1.2.1(a),
$$M(L(G_3^*))=2^{(n_2-1)/2+1}.$$
So we have proved the following:
$$M(L(G_3))=M(L(G_3^*))=2^{(n_2-1)/2+1}. \eqno{(9')}$$
Hence, by $(6)-(8)$ and $(9')$, we have
$$M(L(G))=2^{n/2+1}.$$
\begin{figure}[htbp]
  \centering
 \scalebox{1.1}{\includegraphics{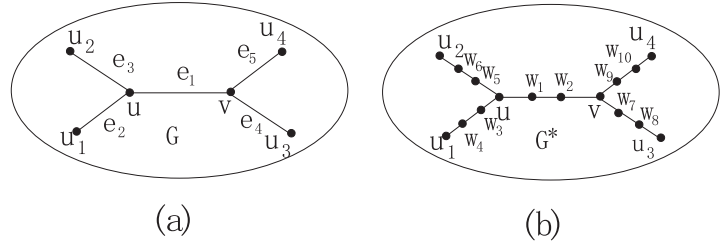}}
  \caption{\ (a)\ A graph $G$;
  \ (b)\ the graph $G^*$;}
\end{figure}

{\bf Subcase 1.2.2} $e_1$ is not a cut edge of $G$.

Construct a new graph $G^*$ from $G$ by replacing five edges $e_1=(u,v),e_2=(u,u_1),e_3=(u,u_2),e_4=(v,u_3)$,
and $e_5=(v,u_4)$ by five paths $(u-w_1-w_2-v), (u-w_3-w_4-u_1), (u-w_5-w_6-u_2), (v-w_7-w_8-u_3)$,
and $(v-w_9-w_{10}-u_4)$, respectively (see Figure 4). That is, $G^*$ is the graph obtained from $G$ by subdividing
each of five edges $e_1,e_2,e_3,e_4$, and $e_5$ twice. By Lemma 2.2,
$$M(L(G))=M(L(G^*)).\eqno{(10)}$$
Let $G_{(1)}$ (resp., $G_{(2)}$) be the graph obtained from $G^*$
by deleting vertices $w_1, w_2$, and $v$, and connecting
vertices $w_7$ to $w_9$ by a new edge $(w_7,w_9)$ (resp., by deleting vertices $w_1, w_2$, and $u$,
and connecting vertices $w_3$ to $w_5$ by a new edge $(w_3,w_5)$). Let $f=(u,w_1),g=(w_1,w_2)$, and $h=(w_2,v)$
(see Figure 4(b)). Note that $M(L(G^*))=M(L(G^*)-f-g)+M(L(G^*)-g-h)$.
With a similar method as in Subcase 1.2.1, we may prove that
$$M(L(G^*)-f-g)=M(L(G_{(1)})), M(L(G^*)-g-h)=M(L(G_{(2)})).$$Hence
$$M(L(G^*))=M(L(G_{(1)}))+M(L(G_{(2)})).\eqno{(11)}$$

Note that since $e_1$ is not a cut edge of $G$, both $G_{(1)}$ and
$G_{(2)}$ are connected graphs with $n-2$ 3-degree vertices. By induction,
$$M(L(G_{(1)}))=M(L(G_{(2)}))=2^{(n-2)/2+1}.\eqno{(12)}$$
From $(10)-(12)$, it follows that $M(L(G))=2^{n/2+1}$.

{\bf Case 2}\ \ $j=1$.

Now $e_1=(u,v_1)$ and $e_2=(v_1,v)$ are two edges of $G$, and $d_{G}(v_1)=2, d_G(u)=d_G(v)=3$.
If $G$ has two multiple edges connecting vertices $u$ and $v$,
then $G$ is the graph with three vertices obtained
from the graph with two vertices and three multiple edges by subdividing
an edge once. It is not difficult to see that $G$ has two 3-degree vertices and  $M(L(G))=4=2^{2/2+1}$.
Now we only need to distinguish the following two subcases:
\begin{figure}[htbp]
  \centering
 \scalebox{1.1}{\includegraphics{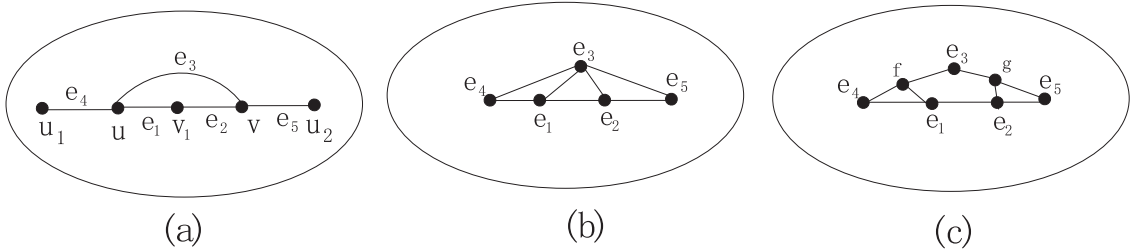}}
  \caption{\ (a)\ A graph $G$;
  \ (b)\ the line graph $L(G)$;
  \ (c)\ the graph $G_1''$.}
\end{figure}

{\bf Subcase 2.1}\ \ There exists one edge $e_3=(u,v)$ in $G$ connecting vertices $u$ and $v$ (see Figure 5(a)).

Let $e_1=(u,v_1),e_3=(u,v)$, and $e_4=(u,u_1)$ (resp., $e_2=(v,v_1), e_3=(u,v)$, and $e_5=(v,u_2)$)
be the three edges incident with vertex $u$ (resp., with verex $v$), and $u_1\neq v, u_2\neq u$ (see Figure 5(a)).
Construct a new graph $G_1'$ from $G$ by deleting vertex $v_1$.
Then $G_1'$ is a connected graph with $n-2$ 3-degree
vertices. By induction,
$$M(L(G_1'))=2^{(n-2)/2+1}.\eqno{(13)}$$
Note that $d_{L(G)}(e_3)=4$ (see Figure 5(b)).
Let $G_1''$ be the graph obtained from $L(G)$ by splitting vertex $e_3$,
which is illustrated in Figure 5(c). Thus, by Lemma 2.1,
$$M(G_1'')=M(L(G)).\eqno{(14)}$$
Note that $M(G_1'')=M(G_1''-f-e_3)+M(G_1''-g-e_3)$ (see Figure 5(c)).
Since each perfect matching of $G_1''-f-e_3$ (resp. $G_1''-g-e_3$) contains no edge $(e_2,e_5)$ (resp. $(e_1,e_4)$),
by Lemma 2.1, it is not difficult to see that $$M(G_1''-f-e_3)=M(G_1''-g-e_3)=M(L(G_1')).$$ Hence, from $(13)$ and $(14)$.
$$M(L(G))=M(G_1'')=2M(L(G_1'))=2^{n/2+1}.$$

{\bf Subcase 2.2}\ \ There exists no edge in $G$ connecting vertices $u$ and $v$.

Using the same method as in Subcase 1.2 (hence we omit the proof), we can show $M(L(G))=2^{n/2+1}$.

So we have finished the proof of the theorem.
\end{proof}

If $G$ is a graph with $n$ vertices ($n\rightarrow \infty$), define the entropy of $G$ as
\cite{Fisher66,Elser89,WW08}

$$\mathcal E(G)=\lim_{n\rightarrow \infty}\frac{2\log(M(G))}{n}.$$

The following result is immediate from Theorem 2.4.
\begin{corollary}
Suppose $G$ is a connected cubic graph $G$ with an even number of edges. Then the number of perfect
matchings of $L(G)$ equals $2^{\nu/2+1}$, and the entropy of $L(G)$ equals $\frac{2\log 2}{3}$,
where $\nu$ is the number of vertices of $G$.
\end{corollary}

Given a connected cubic graph $G$ with $\nu$ vertices, construct a
cubic graph $G'$ with $3\nu$ vertices from $G$ by cutting off all
``corners" of $G$ such that one third of each edge is cut off at
each of both ends (see Figure 6 for an example), which is called
the clique-inserted-graph of $G$ in \cite{ZCC09}. It is not difficult
to see that $G'$ is the line graph $L(S(G))$ of the subdivision
$S(G)$ of $G$. That is, $G'=L(S(G))$.
\begin{figure}[htbp]
  \centering
 \scalebox{0.5}{\includegraphics{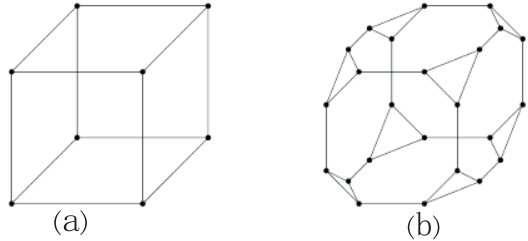}}
  \caption{\ (a)\ A connected cubic graph $G$;
  \ (b)\ the line graph $L(S(G))$.}
\end{figure}
\begin{corollary}
Suppose $G$ is a connected cubic graph $G$ with $\nu$ vertices and
$S(G)$ denotes the graph obtained from $G$ by subdividing every edge
once. Then the number of perfect matchings of the clique-inserted-graph of $G$ (i.e., the line graph of $S(G)$)
equals $2^{\nu/2+1}$, and the entropy of $L(S(G))$ equals $\frac{\log 2}{3}$.
\end{corollary}

\begin{theorem}
Suppose $G$ is a connected cubic line graph $G$ with $\nu$ vertices.
Then

$(1)$. if $G=K_4$, then $M(G)=3$;

$(2)$. if $G\neq K_4$, then there exists a connected cubic graph $G^+$ with $\nu/3$ vertices such that
$G=L(S(G^+))$. Moreover, $M(G)=2^{\nu/6+1}$ and the entropy of $G$ equals $\frac{\log 2}{3}$.
\end{theorem}
\begin{proof}
Suppos that $G'$ is a connected graph and $G=L(G')$ is a connected cubic graph. Let $e=(u,v)$ be an edge of $G'$.
we have
$d_{G'}(u)+d_{G'}(v)-2=d_G(e)=3$, where $d_{G'}(u)$ denotes the degree
of vertex $u$ of $G'$. That is, $d_{G'}(u)+d_{G'}(v)=5$. Then
$(d_{G'}(u),d_{G'}(v))=(1,4),(4,1),(2,3)$ or $(3,2)$.

If $(d_{G'}(u),d_{G'}(v))=(1,4)$ or $(4,1)$, then $G'$ is the star
$K_{1,4}$ with five vertices and $G$ is the complete graph $K_4$
with four vertices (otherwise, $G$ is not a cubic
graph). Hence $M(G)=3$.

If $G'\neq K_{1,4}$, then for every edge $e=(u,v)\in E(G')$ we
have $(d_{G'}(u),d_{G'}(v))=(2,3)$ or $(3,2)$. This implies that $G'$ is a
graph obtained from a connected cubic graph $G^+$ by subdividing every edge once (i.e., $G'=S(G^+)$).
By Theorem 2.4, $M(G)=2^{\nu/6+1}$ and hence the entropy of $G$ equals $\frac{\log 2}{3}$.
\end{proof}
\begin{remark}
By the theorem above, the conjecture posed by Lov\'asz and Plummer holds
for the connected cubic line graphs.
\end{remark}

\section{Applications} \hspace*{\parindent}
As applications, in this section we enumerate perfect matchings of
Sierpinski gasket with dimension two, $3.12.12$ lattices, and Kagom\'e lattices  in the context of statistical
physics.
\begin{figure}[htbp]
  \centering
 \scalebox{0.8}{\includegraphics{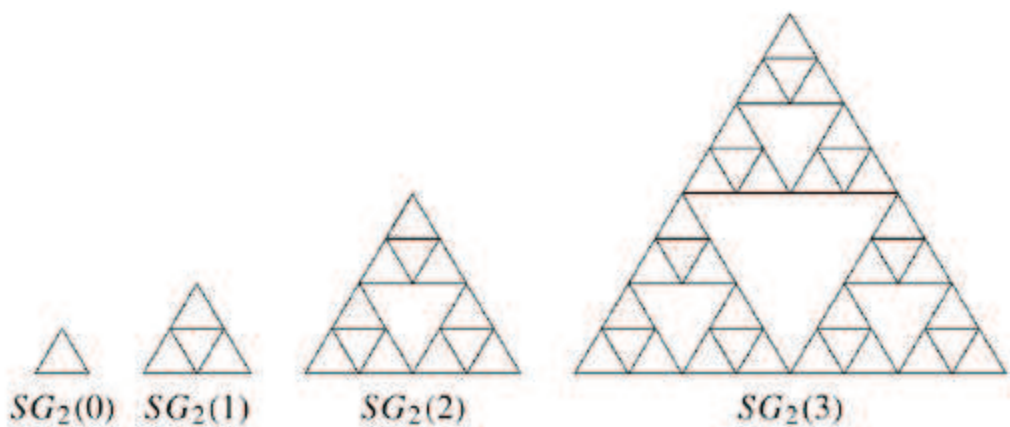}}
  \caption{\ \ The first four stages $n=0,1,2,3$ of two-dimensional Sierpinski gasket $SG_2(n)$.}
\end{figure}\
\subsection{The two-dimensional Sierpinski gasket}\ \
Fractals are geometrical structures of non-integer Hausdorff dimension realized by repeated
construction of an elementary shape on progressively smaller length scales
\cite{GAM81,GMA80,Guyer84,DV98}. A well-known example of fractal
is the Sierpinski gasket which has been extensively studied in several contexts \cite{CC08,GAM81,GMA80,Guyer84,DV98}.
The construction of the two-dimensional Sierpinski gasket, denoted by $SG_2(n)$ at stage $n$ is shown in Figure 7.
At stage $n=0$, it is an equilateral triangle; while stage $n+1$ is obtained by the
juxtaposition of three $n$-stage structures. The two-dimensional Sierpinski gasket has
fractal dimensionality $\ln3/\ln2$ \cite{GAM81}. It is not difficult to see that $SG_2(n)$ has
$\frac{3}{2}(3^n+1)$ vertices and $3^{n+1}$ edges \cite{CC08}.
Now we construct a graph sequence $\{G_n\}_{n\geq 0}$ such that
the line graph of $G_n$ is isomorphic to $SG_2(n)$, $n=0,1,2,\cdots$. In fact,
at stage $n=0$, $G_0$ is a star $K_{1,3}$; while stage $n+1$ is obtained by the
juxtaposition of three $n$-stage structures, see Figure 8. It is easily verified that $L(G_n)=SG_2(n)$.
\begin{figure}[htbp]
  \centering
 \scalebox{0.8}{\includegraphics{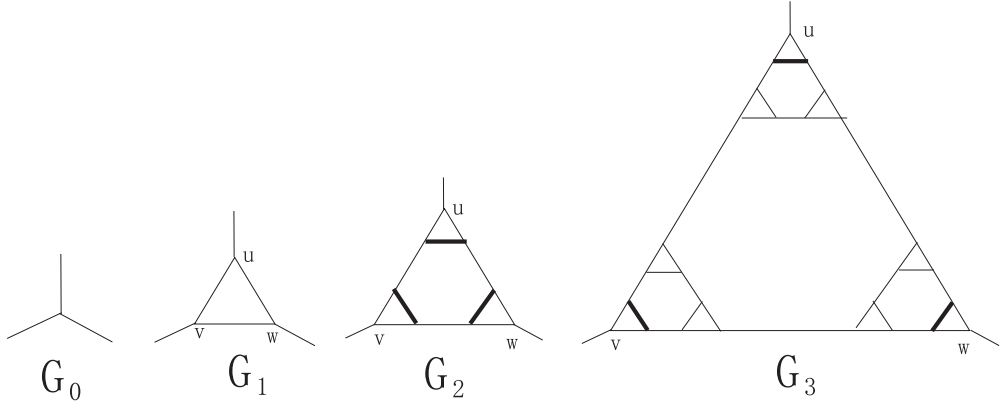}}
  \caption{\ \ The first four stages $n=0,1,2,3$ of the graphs $G_n$.}
\end{figure}

By the definition of $G_n$, $G_n$ has $3^n+3$ vertices, where each of $3^n$ vertices has degree three and each of
three vertices has degree one. The number of edges of $G_n$ equals the number of vertices of $SG_2(n)$, which is
$\frac{3}{2}(3^n+1)$. Obviously, $G_n$ has an even number of edges if $n$ is odd, and $G_n$ has an odd number of edges
otherwise. Since $G_n$ has three vertices of degree one, we can not use directly Theorem 2.4. So we construct a new graph
$G_n'$ from $G_n$ ($n>1$) by deleting the three vertices of degree one and $u,v,w$ illustrated in Figure 8
and replacing each of the three bold edges with two multiple edges. By Claim 1 in the proof of Theorem 2.4,
$$M(L(G_n))=M(L(G_n')).$$
Note that $G_n'$ is a cubic graph with $3^n-3$ vertices which has an even number of edges
if $n$ is odd and an even number of edges otherwise.
Hence, by Theorem 2.4, if $n$ is odd, then
$$M(L(G_n'))=2^{(3^n-3)/2+1}=2^{(3^n-1)/2}.$$

So we give a new method to prove the following:
\begin{theorem}[Chang and Chen, \cite{CC08}]
Suppose $SG_2(n)$ is the two-dimensional Sierpinski gasket. Then the number of perfect matchings of $SG_2(n)$ equals
$2^{(3^n-1)/2}$ if $n$ is odd and zero otherwise. The entropy of $SG_2(n)$ equals
$\frac{2\log 2}{3}$ if $n$ is odd.
\end{theorem}

\begin{figure}[htbp]
  \centering
  \scalebox{0.6}{\includegraphics{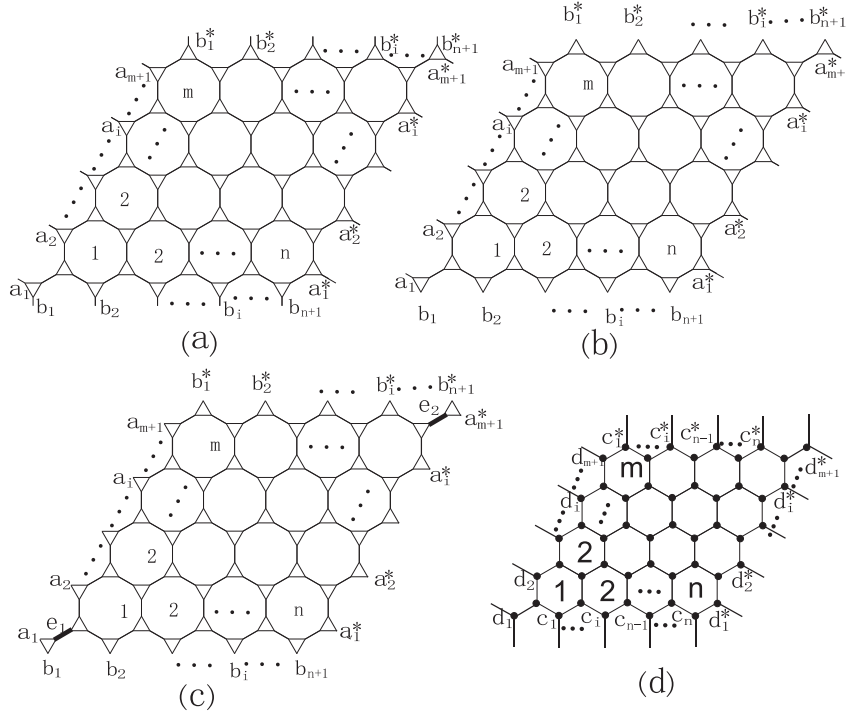}}
  \caption{\ (a)\ \ The $3.12.12$ lattice $R^T(n,m)$ with toroidal boundary
  condition,
  \ (b)\ \ the $3.12.12$ lattice $R^C(n,m)$ with cylindrical boundary condition,
  \ (c)\ \ the $3.12.12$ lattice $R^F(n,m)$ with free boundary condition,
  \ (d)\ \ the hexagonal lattice $H^T(n,m)$ with toroidal boundary condition.}
\end{figure}
\subsection{3.12.12 lattices}
The 3.12.12 lattice $R^T(n,m)$ with toroidal boundary condition is
shown in Figure 9(a), where
$(a_1,a_1^*),(a_2,a_2^*),\ldots,(a_{m+1},a_{m+1}^*)$, and
$(b_1,b_1^*),(b_2,b_2^*),\ldots,\\(b_{n+1},b_{n+1}^*)$ are edges in
$R^T(n,m)$. The 3.12.12 lattice $R^T(n,m)$ has been used by Fisher
\cite{Fisher66} in a dimer formulation of the Ising model. If we
delete edges $(b_1,b_1^*),(b_2,b_2^*),\ldots,(b_{n+1},b_{n+1}^*)$
from $R^T(n,m)$, the $3.12.12$ lattice $R^C(n,m)$ with cylindrical
boundary condition is obtained (see Figure 9(b)). If we
delete edges $(a_1,a_1^*),(a_2,a_2^*),\ldots,(a_{m+1},a_{m+1}^*)$
from $R^C(n,m)$, the $3.12.12$ lattice $R^F(n,m)$ with free
boundary condition is obtained (see Figure 9(c)). By means of Pfaffians,
Fisher \cite{Fisher66} and Wu \cite{Wu06} proved that the logarithm of
the number of perfect matchings of $R^T(n,m)$, divided by $3(m+1)(n+1)$ (the number of edges of
each of perfect matchings of $R^T(n,m)$), converges $\frac{1}{3}\ln 2$ as $m,n\rightarrow \infty$, which
is called the entropy of $R^T(n,m)$ by statistial physicists. By Theorem 2.4, we derive the exact formulae
of the numbers of perfect matchings of $R^T(n,m), R^C(n,m)$, and $R^F(n,m)$ as follows.
\begin{theorem}
Let $R^T(n,m), R^C(n,m)$, and $R^F(n,m)$ be the $3.12.12$ lattices with toroidal, cylindrical, and free boundary
conditions, respectively. Then
$$M(R^T(n,m))=2^{mn+m+n+2},$$$$M(R^C(n,m))=2^{mn+m+1},$$$$M(R^F(n,m))=2^{mn}.$$
Hence $R^T(n,m), R^C(n,m)$, and $R^F(n,m)$ have the same entropy $\frac{1}{3}\ln 2$.
\end{theorem}
\begin{proof}
In order to prove the theorem, we introduce the hexagonal lattices which have been extensively
studied by statistical physicists \cite{Elser89, Fisher66,Wu06}.
The hexagonal lattice $H^T(n,m)$ with toroidal boundary condition
is shown in Figure 9(d), where
$(d_1,d_1^*),(d_2,d_2^*),\ldots,(d_{m+1},d_{m+1}^*)$ and
$(d_1,c_1^*),(c_1,c_2^*),\ldots,(c_{n-1},c_{n}^*),(c_n,d_{m+1}^*)$ are edges in
$H^T(n,m)$. It is not difficult to see that the line graph of $S(H^T(n,m))$ is $R^T(n,m)$, where $S(H^T(n,m))$
is the graph obtained from
$H^T(n,m)$ by subdividing each edge of $H^T(n,m)$ once. Note that there exists $2(m+1)(n+1)$
vertices of degree three in $S(H^T(n,m))$. By Theorem 2.4,
$$M(R^T(n,m))=M(L(S(H^T(n,m))))=2^{2(m+1)(n+1)/2+1}=2^{mn+m+n+2}.$$
\begin{figure}[htbp]
  \centering
  \scalebox{0.7}{\includegraphics{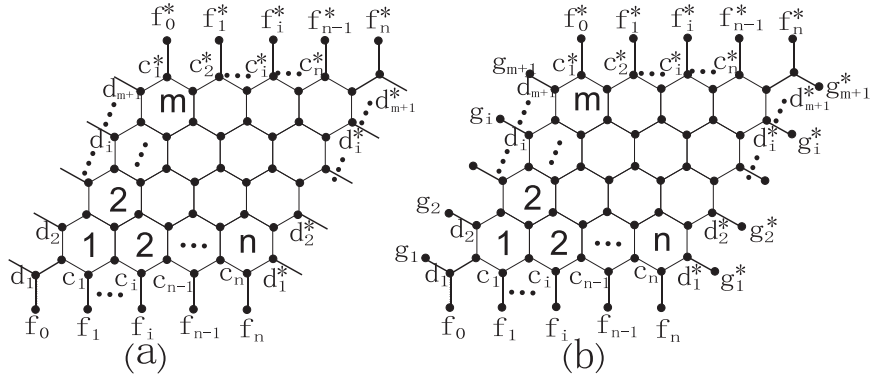}}
  \caption{\ (a)\ \ The graph $H^T_1(n,m)$,
  \ (b)\ \ the graph $H^T_2(n,m)$.}
\end{figure}

Let $H^T_1(n,m)$ be the graph obtained from $H^T(n,m)$ by replacing each of the $n+1$ edges
$(d_1,c_1^*),(c_1,c_2^*),\ldots,(c_{n-1},c_{n}^*),(c_n,d_{m+1}^*)$ with two independent edges
(i.e., replacing edge $(d_1,c_1^*)$ with edges $(d_1,f_0)$ and $(c_1^*,f_0^*)$,
replacing edge $(c_1,c_2^*)$ with edges $(c_1,f_1)$ and $(c_2^*,f_1^*)$, $\cdots$,
replacing edge $(c_{n-1},c_n^*)$ with edges $(c_{n-1},f_{n-1})$ and $(c_n^*,f_{n-1}^*)$, and replacing
edge $(c_n,d_{m+1}^*)$ by edges $(c_n,f_n)$ and $(d_{m+1}^*,f_n^*)$, see Figure 10(a)).
For the graph $H^T_1(n,m)$, subdivide each edge
in $E(H^T_1(n,m))\backslash A$ once, where $E(H^T_1(n,m))$ is the edge set of $H^T_1(n,m)$ and
$A$ is the set of $2n+2$ edges $(d_1,f_0),(c_1^*,f_0^*),(c_1,f_1),(c_2^*,f_1^*),\cdots,(c_{n-1},f_{n-1}),\\
(c_n^*,f_{n-1}^*),(c_n,f_n),(d_{m+1}^*,f_n^*)$. The resulting graph is denoted by $H^{\star}_1(n,m)$. Obviously,
$R^C(n,m)$ is the line graph of $H^{\star}_1(n,m)$. Note that $H^{\star}_1(n,m)$ has
$2(m+1)(n+1)=2mn+2m+2n+2$ vertices of degree three and $2n+2$ vertices of degree one, and
the degree of each of other vertices of $H^{\star}_1(n,m)$ is two.
By Theorem 2.4 and Claim 1 in the proof of Theorem 2.4,
$$M(R^C(n,m))=2^{(2mn+2m+2n+2-2n-2)/2+1}=2^{mn+m+1}.$$

Let $H^T_2(n,m)$ be the graph obtained from $H^T_1(n,m)$ by replacing each of the $m+1$ edges
$(d_1,d_1^*),(d_2,d_2^*),\ldots,(d_{m+1},d_{m+1}^*)$ with two independent edges
(i.e., replacing edge $(d_1,d_1^*)$ with edges $(d_1,g_1)$ and $(d_1^*,g_1^*)$,
replacing edge $(d_2,d_2^*)$ with edges $(d_2,g_2)$ and $(d_2^*,g_2^*)$, $\cdots$,
replacing edge $(d_{m+1},d_{m+1}^*)$ with edges $(d_{m+1},g_{m+1})$ and $(d_{m+1}^*,g_{m+1}^*)$, see Figure 10(b)).
Let $B=\{(d_i,g_i),(d_i^*,g_i^*)|i=1,2,\cdots,m+1\}$. For the graph $H^T_2(n,m)$, subdivide each edge
in $E(H^T_2(n,m))\backslash (A\cup B)$ once. The resulting graph is deonted $H_2^{\star}(n,m)$. It is not difficult to
see that the line graph of $H_2^{\star}(n,m)$ is just $R^F(n,m)$, i.e., $R^F(n,m)=L(H_2^{\star}(n,m))$.

Let $H_3^T(n,m)$ be the graph obtained from $H_2^T(n,m)$ by deleting six vertices $f_0,d_1,g_1,f_n^*,\\d_{m+1}^*$ and
$g_{m+1}^*$ illustrated in Figure 10(b). Denote by $H_3^{\star}(n,m)$ the graph obtained by subdividing each edge of
$H_3^t(n,m)$ once which is not a pendent edge. Hence the number of 3-degree
(resp., 1-degree) verteices of $H_3^{\star}(n,m)$ equals
$2(m+1)(n+1)-4=2mn+2m+2n-2$ (resp., 2m+2n). By Theorem 2.4 and Claim 1 in the proof of Theorem 2.4,
$$M(L(H_3^{\star}(n,m)))=2^{(2mn+2m+2n-2-2m-2n)/2+1}=2^{mn}.\eqno{(15)}$$

Note that every perfect matching of $R^F(n,m)$ contains edges $(a_1,b_1),(a_{m+1},b_{n+1})$
and the two bald edges $e_1$ and $e_2$ illustrated in Figure 9(c).
Let $R^F_1(n,m)$ be the graph obtained from $R^F(n,m)$ by deleting the eight vertices
incident with edges $e_1,e_2, (a_1,b_1),(a_{m+1},b_{n+1})$. Then
$$M(R^F(n,m))=M(R^F_1(n,m)).\eqno{(16)}$$

Obviously, the line graph of $H_3^{\star}(n,m)$ is $R^F_1(n,m)$.
Thus
$$M(R^F_1(n,m))=M(L(H_3^{\star}(n,m))).\eqno{(17)}$$
By $(15),(16)$, and $(17)$,
$$M(R^F(n,m))=2^{mn}.$$

Hence
$$\lim_{m,n\rightarrow \infty}\frac{\ln {M(R^T(n,m))}}{3(m+1)(n+1)}=
\lim_{m,n\rightarrow \infty}\frac{\ln {M(R^C(n,m))}}{3(m+1)(n+1)}=
\lim_{m,n\rightarrow \infty}\frac{\ln {M(R^F(n,m))}}{3(m+1)(n+1)}=\frac{1}{3}\ln 2, $$
implying that Theorem 3.2 holds.
\end{proof}
\begin{remark}
Similary, we can define the $3.3.12$ lattices $R^K(n,m)$ and $R^M(n,m)$ with Klein-bottle and Mobius-band boundary
conditions, respectively. It is similarly verified that
$$M(R^K(n,m))=2^{mn+m+n+2}, M(R^M(n,m))=2^{mn+m+1}.$$
\end{remark}
\begin{figure}[htbp]
  \centering
  \scalebox{0.7}{\includegraphics{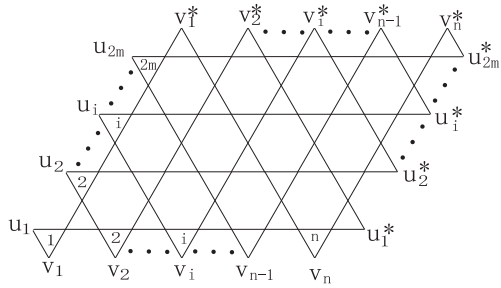}}
  \caption{ The graph $G(n,m)$.}
\end{figure}

\subsection{Kagom\'e lattices}
Let $G(n,m)$ be the plane lattice graph illustrated in Figure 11,
each of whose vertices has degree two or four. For $G(n,m)$,
if we identify each
pair of vertices $u_i$ and $u_i^*$, $v_j$ and $v_j^*$, $i=1,2,\cdots,2m, j=1,2,\cdots,n$, the resulting graph,
denoted by $K^T(n,m)$, is
called the Kagom\'e lattice with toroidal boundary condition by statistical physicists
(see \cite{Elser89,PW88,Wu06,WW07,WW08}). For $G(n,m)$, if we delete
vertices $v_1^*,v_2^*,\cdots,v_n^*$ and identify each pair of vertices $u_i$ and $u_i^*$, $i=1,2,\cdots,2m$,
the resulting graph, denoted by $K^C(n,m)$, is called the Kagom\'e lattice with cylindrical boundary condition
(see \cite{WW08}). And
the graph obtained from $G(n,m)$ by deleting vertices $u_i^*,v_j^*, i=1,2,\cdots,2m,j=1,2,\cdots,n$, is called
the Kagom\'e lattice with free boundary condition, denoted by $K^F(n,m)$. By the definitions of $K^T(n,m), K^C(n,m)$,
and $K^F(n,m)$, we know that all $K^T(n,m), K^C(n,m)$,
and $K^F(n,m)$ have $6mn$ vertices.

The study of the molecular freedom for the kagom\'e lattice has been a
subject matter of interest for many years (see, for example, \cite{Elser89,PW88}), but
most of the studies have been numerical or approximate. By using Paffian orientation,
Wu and Wang \cite{WW08} obtained the interesting formulae of the numbers of
perfect matchings of $K^T(n,m)$ and $K^C(n,m)$ as follows:
$$M(K^T(n,m))=2^{2mn+1},\ \ M(K^C(n,m))=2^{2mn-n+1}.\eqno{(18)}$$
In fact, they gave a more general formulae (each edge was weighted, see \cite{WW08}). Now we give a new method to
prove $(18)$. Furthermore, we will prove the following:
$$M(K^F(n,m))=2^{2mn-2m-n+1}.\eqno{(19)}$$
\begin{figure}[htbp]
  \centering
  \scalebox{0.8}{\includegraphics{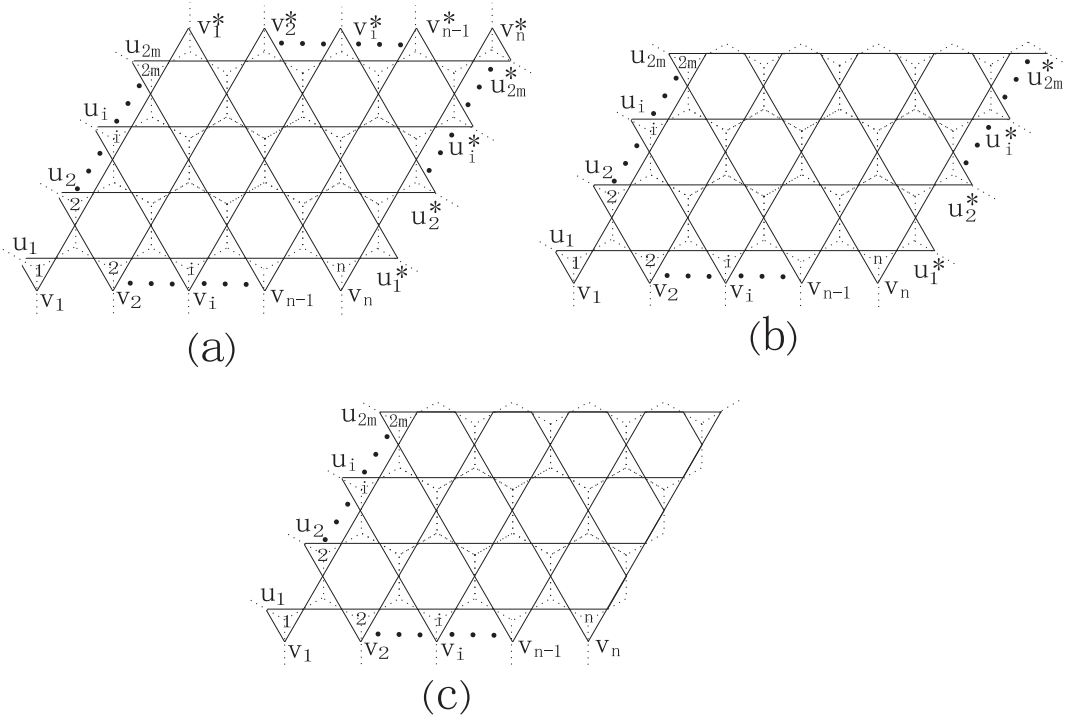}}
  \caption{\ (a)\ \ The graph $K^T(n,m)$, where $u_i$ and $u_i^*$ are identified as a single vertex, $i=1,2,\cdots,2m$,
  and $v_j$ and $v_j^*$ are identified as a single vertex, $j=1,2,\cdots,n$.
  \ (b)\ \ The graph $K^C(n,m)$,  where $u_i$ and $u_i^*$ are identified as a single vertex, $i=1,2,\cdots,2m$.
  \ (c)\ \ The graph $K^F(n,m)$.}
\end{figure}

In the proof of Theorem 3.2, we have defined the hexagonal lattice $H^T(n,m)$ with toroidal boundary condition (see
Figure 9(d)) and two graphs $H_1^T(n,m)$ and $H_2^T(n,m)$ (see Figure 10). It is not difficult to
see that $K^T(n,m)$ is the line graph of $H^T(n-1,2m-1)$ (see Figure 12(a), where
we embed $H^T(n-1,2m-1)$ and $K^T(n,m)$ simultaneously in the plane), $K^C(n,m)$ is the line graph of the graph
, denoted by $H_1(n-1,2m-1)$, which is obtained from $H_1^T(n-1,2m-1)$
by deleting vertices $f_0^*,f_1^*,\cdots,f_{n-1}^*$ (see Figure 12(b), where
we embed $K^C(n,m)$ and $H_1(n-1,2m-1)$ simultaneously in the plane),
and $K^F(n,m)$ is the line graph of the graph, denoted by $H_2(n-1,2m-1)$, which is
obtained from $H_2^T(n-1,2m-1)$ by deleting vertices $f_0^*,f_1^*,\cdots,f_{n-1}^*,g_1^*,g_2^*,\cdots,g_{2m}^*$
(see Figure 12(c), where we embed $K^F(n,m)$ and $H_2(n-1,2m-1)$ simultaneously in the plane),
respectively. With the same
method as in the proof of Theorem 3.2, we can prove $(18)$ and $(19)$. Hence we have the following:

\begin{theorem}
Let $K^T(n,m), K^C(n,m)$, and $K^F(n,m)$ be the Kagom\'e lattices with toroidal, cylindrical, and free boundary
conditions, respectively. Then
$M(K^T(n,m))=2^{2mn+1},M(K^C(n,m))=2^{2mn-n+1}, M(K^F(n,m))=2^{2mn-2m-n+1}.$
Hence $K^T(n,m), K^C(n,m)$, and $K^F(n,m)$ have the same entropy $\frac{2}{3}\ln 2$.
\end{theorem}
\begin{remark}
Similary, we can define the Kagom\'e lattices $K^K(n,m)$ and $K^M(n,m)$ with Klein-bottle and Mobius-band boundary
conditions, respectively. It is similarly verified that
$$M(K^K(n,m))=2^{2mn+1}, M(K^M(n,m))=2^{2mn-n+1}.$$
\end{remark}

\section{CONCLUDING REMARKS}
Kuperberg \cite{Kup98} showed that the number of perfect matchings of
the line graph of a graph with vertices of degree at most 3 (and
with an even number of edges) is a power of 2. In this paper, we obtain the exact formula of
the number of perfect matchings of the line graph of a graph with vertices of degree equal to
two or three and with an even number of edges. Moreover, our result implies that
the conjecture of Lov\'asz and Plummer on the
perfect matchings of regular graphs holds for the connected cubic line graphs
and the line graphs of connected cubic graphs with an even number of edges.
Finally, as applications of our result, we use a unified method to prove some known formulae of
perfect matchings of the Kagom\'e lattices
with toroidal and cylindrical boundary conditions
by Wang and Wu \cite{WW07,WW08}, the $3.12.12$ lattices with toroidal
boundary condition by Fisher \cite{Fisher66} and Wu \cite{Wu06}, and
the Sierpinski gasket with dimension two by Chang and Chen \cite{CC08},
respectively. Furthermore, by using this unified approach
we solve the problem of enumeration of perfect matchings of the Kagom\'e lattices
with free, Klein-bottle, and Mobius-band boundary conditions, and the $3.12.12$ lattices with cylindrical,
free, Klein-bottle, and Mobius-band boundary conditions, respectively.


\end{document}